\documentclass[a4paper,11pt]{article}
\usepackage{latexsym}
\usepackage{amsthm}
\usepackage{amsmath}
\usepackage{graphicx}
\usepackage{epsfig}
\usepackage{mathrsfs}
\usepackage[numbers,sort&compress]{natbib}

\newtheorem{theorem}{Theorem}[section]

\newtheorem{lemma}[theorem]{Lemma}

\newtheorem{remark}[theorem]{Remark}
\newtheorem{corollary}[theorem]{Corollary}
\date{}

\begin{document}
\title{Permanental polynomials of skew adjacency matrices of oriented graphs\thanks{This work is supported by NSFC (Grant No. 11501050) and the Fundamental Research Funds for the Central Universities (Grant Nos. 310812151003, 310812152002).}}

\author{Shunyi Liu\thanks{Corresponding author.}\\
{\small College of Science, Chang'an University}\\
{\small Xi'an, Shaanxi 710064, P.R. China}\\
{\small E-mail: liu@chd.edu.cn}\\
\\
Heping Zhang\\
{\small School of Mathematics and Statistics, Lanzhou University}\\
{\small Lanzhou, Gansu 730000, P.R. China}\\
{\small E-mail: zhanghp@lzu.edu.cn}}

\maketitle

\begin{abstract}
Let $G^\sigma$ be an orientation of a simple graph $G$. In this paper, the permanental polynomial of an oriented graph $G^\sigma$ is introduced. The coefficients of the permanental polynomial of $G^\sigma$ are interpreted in terms of the graph structure of $G^\sigma$, and it is proved that all orientations $G^\sigma$ of $G$ have the same permanental polynomial if and only if $G$ has no even cycles. Furthermore, the roots of the permanental polynomial of $G^\sigma$ are studied.
\end{abstract}
\vspace{2mm}



\section{Introduction}\label{intro-sec}

Let $G$ be a simple graph with vertex set $V(G)$ = $\{v_1,\dots,v_n\}$ and edge set $E(G)$. The adjacency matrix of $G$ is the $n \times n$ matrix $A(G)=(a_{ij})$, where $a_{ij}=1$ if $\{v_i,v_j\}\in E(G)$ and $a_{ij}=0$ otherwise. An \emph{oriented graph} $G^\sigma$ is a simple graph $G$ with an orientation $\sigma$, which assigns to each edge a direction so that $G^\sigma$ becomes a directed graph. Both $\sigma$ and $G^\sigma$ are called \emph{orientations} of $G$. The \emph{skew adjacency matrix} of $G^\sigma$, denoted by $A_s(G^\sigma)$, is defined to be the $n\times n$ matrix $(a_{ij})$ whose ($i$,$j$)-entry $a_{ij}$ satisfies
\[
 a_{ij}=
 \left\{
  \begin{array}{ll}
   1,  &\mbox{if $(v_i,v_j)\in E(G^\sigma)$,}\\
   -1, &\mbox{if $(v_j,v_i)\in E(G^\sigma)$,}\\
   0,  &\mbox{otherwise.}
  \end{array}
 \right.
\]

It is known that the skew adjacency matrix plays an important role in enumeration of perfect matchings, since the square of the number of perfect matchings of a graph $G$ with a Pfaffian orientation $G^\sigma$ is equal to the determinant of $A_s(G^\sigma)$ (see~\cite{LoPl86} for more details). There are few studies on skew adjacency matrices of oriented graphs. Recently, the characteristic polynomial of $A_s(G^\sigma)$ was considered~\cite{CaCF12}. In particular, the skew energy~\cite{AdBS10} of an oriented graph $G^\sigma$ (the sum of the absolute values of the eigenvalues of $A_s(G^\sigma)$) has attracted a lot of interest of researchers. Analogously, we will define the permanental polynomial of an oriented graph and investigate the coefficients and roots of this polynomial.

The \emph{permanent} of an $n\times n$ matrix $M$ with entries
$m_{ij}$ $(i,j=1, 2,\dots, n)$ is defined by
\[
\mathrm{per}(M) = \sum_{\pi\in S_n}\prod_{i=1}^{n}m_{i\pi(i)},
\]
where $S_n$ is the symmetric group on $n$ elements. The permanent is defined similarly to the determinant. However, no efficient algorithm for computing the permanent is known, while the determinant can be calculated using Gaussian elimination. More precisely, Valiant~\cite{Val79} has shown that computing the permanent is $\#$P-complete even when restricted to (0,1)-matrices. Permanents have important applications in combinatorics and graph theory in connection with certain enumeration and extremal problems.

Let $G^\sigma$ be an orientation of a simple graph $G$ and $A_s(G^\sigma)$ the skew adjacency matrix of $G^\sigma$. The \emph{permanental polynomial} of $G^\sigma$, denoted by $\pi(G^\sigma,x)$, is defined as the permanent of the characteristic matrix of $A_s(G^\sigma)$, i.e.,
\[
\pi(G^\sigma,x) = \mathrm{per}(xI-A_s(G^\sigma)).
\]
The roots of $\pi(G^\sigma,x)$ are called the \emph{permanental roots} of $G^\sigma$.

The \emph{permanental polynomial} of a graph $G$, denoted by $\pi(G,x)$, is defined by $\mathrm{per}(xI-A(G))$, and the roots of $\pi(G,x)$ are called the \emph{permanental roots} of $G$. The permanental polynomial of a graph was first studied in mathematical literature by Merris et al.~\cite{MeRW81}, and the study of this polynomial in chemical literature was started by Kasum et al.~\cite{KaTG81}. Yan and Zhang~\cite{YaZh04} proved that for a bipartite graph $G$ containing no even subdivisions of $K_{2,3}$, there exists an orientation $G^\sigma$ of $G$ such that $\pi(G,x)$ = $\det(xI-A_s(G^\sigma))$.

The rest of the paper is organized as follows. In Section~\ref{S:coefficients}, we obtain the Sachs form coefficient formula of $\pi(G^\sigma,x)$, which expresses the coefficients of $\pi(G^\sigma,x)$ in terms of the graph structure of $G^\sigma$. We generalize this formula to weighted oriented graphs, and provide a graph-theoretic method to compute the permanent and permanental polynomial of a general skew symmetric matrix. Section~\ref{S:characterization} gives a characterization on graphs whose all orientations have the same permanental polynomial. In Section~\ref{S:permanental_root}, the properties of roots of $\pi(G^\sigma,x)$ are studied.

\section{Coefficients of the permanental polynomial}\label{S:coefficients}
In this section, we obtain the coefficients of the permanental polynomial of an oriented graph in terms of the graph structure, and generalize this result to weighted oriented graphs.

Let $S_n$ be the symmetric group on $n$ elements. It is well-known that every permutation in $S_n$ is a product of disjoint cycles. Denote by $\mathcal{E}(n)$ the set of all permutations in $S_n$ with all cycles having even length.

\begin{lemma}\label{L:skew_permanent}
Let $A=(a_{ij})$ be an $n\times n$ skew symmetric matrix. Then
\[
\mathrm{per}(A)=\sum_{\pi\in \mathcal {E}(n)}a_{1\pi(1)}a_{2\pi(2)}\dots a_{n\pi(n)}.
\]
\end{lemma}

\begin{proof}
Let $\pi=r_1r_2\dots r_t$ be a permutation in $S_n\backslash\mathcal{E}(n)$, where $r_1$, $r_2$, $\dots$, $r_t$ are the disjoint cycles of $\pi$. Since $A$ is a skew symmetric matrix, $a_{ii}$ = 0 for $i = 1, 2, \dots, n$. Thus $\pi$ will contribute 0 to $\mathrm{per}A$ if $\pi$ has a fixed point. So we assume that $\pi$ is fixed-point-free. We define the least element of a cycle $r_i$ of $\pi$ to be the least element of \{1,2,\dots,$n$\} in $r_i$. Since $\pi\in S_n\backslash\mathcal{E}(n)$, $\pi$ contains an odd cycle. We obtain $\pi^\prime$ from $\pi$ by only reversing the odd cycle with the smallest least element. It is easy to see that $(\pi^{\prime})^{\prime}=\pi$ and $\prod_{i=1}^n a_{i\pi(i)}$ = $-\prod_{i=1}^n a_{i\pi^{\prime}(i)}$. Thus we have partitioned the fixed-point-free permutations in $S_n\backslash\mathcal{E}(n)$ into pairs such that each pair contributes 0 to $\mathrm{per}(A)$. This completes the proof.
\end{proof}

Let $G^\sigma$ be an orientation of a graph $G$ and $C$ an even cycle in $G$. We say that $C$ is {\em oddly} (resp. {\em evenly}) {\em oriented} if for either choice of direction of traversal around $C$, the number of oriented edges of $C$ whose orientation agrees with the direction of traversal is odd (resp. even). A {\em Sachs graph} is an undirected graph in which each component is a single edge or a cycle.

\begin{theorem}\label{T:coeffient_skew}
Suppose $\pi(G^\sigma,x)$ = $\sum_{k=0}^{n}a_{k}x^{n-k}$. Then
\[
  a_k=\sum_{U\in\mathcal{E}\mathscr{U}_k}(-1)^{m(U)+c^{-}(U)}\,2^{c(U)}, \quad \mbox{if $k$ is even},
\]
and $a_k=0$ otherwise, where $\mathcal{E}\mathscr{U}_k$ is the set of all Sachs subgraphs of $G$ on $k$ vertices with no odd cycles, $m(U)$ is the number of single edges of $U$, $c(U)$ is the number of cycles of $U$ and $c^{-}(U)$ is the number of oddly oriented cycles of $U$ relative to $G^\sigma$.
\end{theorem}

\begin{proof}
Let us first consider the permanent of $A_s(G^\sigma)$. By Lemma~\ref{L:skew_permanent}, we have
\[
\mathrm{per}(A_s(G^{\sigma}))=\sum_{\pi\in \mathcal {E}(n)}a_{1\pi(1)}a_{2\pi(2)}\dots a_{n\pi(n)}.
\]

Let $\pi$ = $r_1r_2\dots r_t$ be a permutation in $\mathcal {E}(n)$. The term $a_{1\pi(1)}a_{2\pi(2)}\dots a_{n\pi(n)}$ is nonzero if and only if $a_{i\pi(i)}\neq0$ for $i=1,2,\dots,n$, i.e., $(v_i,v_{\pi(i)})$ or $(v_{\pi(i)},v_i)$ is an arc of $G^\sigma$. If $a_{1\pi(1)}a_{2\pi(2)}\dots a_{n\pi(n)}\neq0$, then this term determines a Sachs subgraph $U\in\mathcal{E}\mathscr{U}_n$ in which the components isomorphic to the complete graph $K_2$ are determined by the transpositions among the $r_i$, and the even cycles are determined by the remaining $r_i$. Conversely, $U$ arises from $2^{c(U)}$ permutations, namely $r_1^{\pm1}r_2^{\pm1}\dots r_{c(U)}^{\pm1}r_{c(U)+1}\dots r_{t}$, where $r_1,r_2,\dots,r_{c(U)}$ are the $r_i$ of length greater than 2. It is easy to see that a single edge contributes $-1$, an oddly oriented even cycle contributes $-1$, and an evenly oriented even cycle contributes 1 to the term $a_{1\pi(1)}a_{2\pi(2)}\dots a_{n\pi(n)}$, respectively. Thus we have
\[
\mathrm{per}(A_s(G^{\sigma}))=\sum_{U\in\mathcal{E}\mathscr{U}_n}(-1)^{m(U)+c^{-}(U)}2^{c(U)}.
\]

If $n$ = $\left|V(G)\right|$ is odd, then $\mathcal {E}(n)$ = $\emptyset$, and consequently $\mathrm{per}(A_s(G^{\sigma}))$ = 0.

It is known that $(-1)^{k}a_k$ equals the sum of all $k\times k$ principal subpermanents of $A_s(G^{\sigma})$. Note that there is a one-to-one correspondence between the set of these principal subpermanents and the set of all induced subgraphs of $G$ having exactly $k$ vertices. Applying the result obtained above to each of the $n \choose k$ principal subpermanents and summing, we have
\[
  (-1)^{k}a_k=\sum_{U\in\mathcal{E}\mathscr{U}_k}(-1)^{m(U)+c^{-}(U)}\,2^{c(U)}.
\]
Therefore, $a_k=\sum_{U\in\mathcal{E}\mathscr{U}_k}(-1)^{m(U)+c^{-}(U)}\,2^{c(U)}$ if $k$ is even, and $a_k=0$ otherwise.
\end{proof}

Theorem~\ref{T:coeffient_skew} can be extended to weighted oriented graphs. Suppose that $G$ is a simple graph with vertex set $V(G)$ = $\{v_1,v_2,\dots,v_n\}$. Let $G_\omega^\sigma$ be an oriented graph $G^{\sigma}$ with a weight function $\omega$, which assigns to each arc $(v_i,v_j)$ a weight $\omega_{ij}$ so that $G_\omega^\sigma$ becomes a weighted oriented graph. The \emph{generalized skew adjacency matrix} of $G_\omega^\sigma$  is defined to be the $n\times n$ matrix $A_s(G_\omega^\sigma)=(a_{ij})$, whose $(i,j)$-entry $a_{ij}$ satisfies
\[
 a_{ij}=
 \left\{
  \begin{array}{ll}
   \omega_{ij}, &\mbox{if $(v_i,v_j)\in E(G_\omega^\sigma)$,}\\
   -\omega_{ji}, &\mbox{if $(v_j,v_i)\in E(G_\omega^\sigma)$,}\\
   0, &\mbox{otherwise.}
  \end{array}
 \right.
\]

Following similar arguments as in the proof of Theorem \ref{T:coeffient_skew}, we can obtain the coefficients of the permenental polynomial of a weighted oriented graph. If $G$ is a weighted undirected graph and $U$ is a Sachs subgraph of $G$, let
\[
\prod(U)=\prod_{e\in E(U)}(w(e))^{\zeta(e;U)},
\]
where $E(U)$ is the set of edges of $U$, $w(e)$ is the weight of $e$, and
\[
\zeta(e;U)=
 \left\{
  \begin{array}{ll}
   1, &\mbox{if $e$ is contained in some cycle of $U$,}\\
   2, &\mbox{otherwise.}
  \end{array}
 \right.
\]

\begin{theorem}\label{T:coeffient_skew_weight}
Let $G_\omega^\sigma$ be a weighted oriented graph of a simple graph $G$ and $A_s(G_\omega^\sigma)$ the generalized skew adjacency matrix of $G_\omega^\sigma$. Suppose $\pi(G_\omega^\sigma,x)$ = $\mathrm{per}(xI-A_s(G_\omega^\sigma))$ = $\sum_{k=0}^{n}a_{k}x^{n-k}$. Then $a_k$ = 0 if $k$ is odd, and
\begin{equation}\label{E:poly_weight_oriented}
  a_k = \sum_{U\in\mathcal{E}\mathscr{U}_k}(-1)^{m(U)+c^{-}(U)}2^{c(U)}\prod(U),\quad \mbox{if $k$ is even,}
\end{equation}
where $\mathcal{E}\mathscr{U}_k$, $m(U)$, $c(U)$ and $c^{-}(U)$ are defined as in Theorem \ref{T:coeffient_skew}. In particular, we have
\begin{equation}\label{E:per_weight_oriented}
  \mathrm{per}(A_s(G_\omega^\sigma))=\sum_{U\in\mathcal{E}\mathscr{U}_n}(-1)^{m(U)+c^{-}(U)}2^{c(U)}\prod(U).
\end{equation}
\end{theorem}

As an application of Theorem~\ref{T:coeffient_skew_weight}, we present a graph-theoretic method to calculate the permanent and permanental polynomial of a skew symmetric matrix. Suppose that $A=(a_{ij})$ is an $n\times n$ real skew symmetric matrix. We can construct a weighted oriented graph $G_\omega^\sigma$ associated to $A$ as follows: let $\{v_1,v_2,\dots,v_n\}$ be the vertex set of $G_\omega^\sigma$ such that $v_i$ corresponds with the $i$-th row (and the corresponding column) of $A$, $(v_i,v_j)\in E(G_\omega^\sigma)$ if and only if $a_{ij}>0$, and we assign weight $a_{ij}$ to the arc $(v_i,v_j)$. An example is illustrated in Fig.~\ref{F:figure_1}.

\begin{figure}[tphb]
\begin{center}
\includegraphics[scale=0.6]{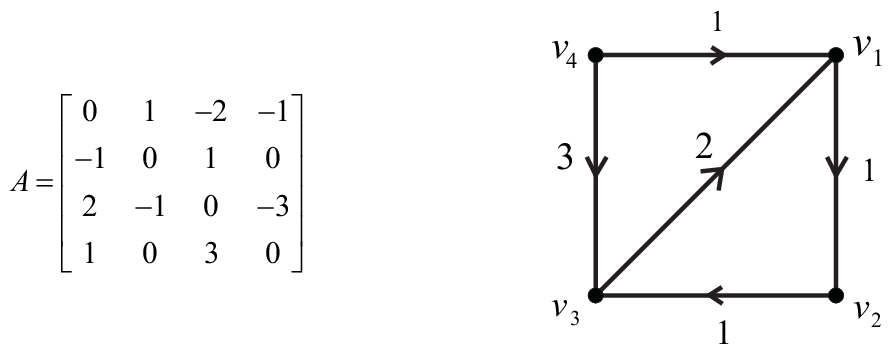}
\end{center}
\caption{A skew symmetric matrix $A$ and its associated weighted oriented graph $G_\omega^\sigma$.}
\label{F:figure_1}
\end{figure}

Clearly, the weighted oriented graph $G_\omega^\sigma$ associated to $A$ has the generalized skew adjacency matrix equal to $A$. Therefore, by Theorem~\ref{T:coeffient_skew_weight}, we can calculate the permanental polynomial of an arbitrary real skew symmetric matrix $A$ considered as the generalized skew adjacency matrix of a weighted oriented graph $G_\omega^\sigma$ (associated to $A$). For example, applying (\ref{E:poly_weight_oriented}) and (\ref{E:per_weight_oriented}) to the weighted oriented graph $G_\omega^\sigma$ associated to the $4\times4$ matrix $A$ in Fig.~\ref{F:figure_1}, we obtain $\mathrm{per}(xI_4-A)$ = $x^4-16x^2+4$ and $\mathrm{per}A=4$. It is not difficult to see that the above method can be modified to compute the permanent and permanental polynomial of a complex skew symmetric matrix.

\section{Graphs whose all orientations have the same permanental polynomial}\label{S:characterization}
Let $G$ be a simple graph with $m$ edges. Since each edge has two possible directions, it follows that $G$ has $2^m$ distinct orientations. It is of interest to know whether all the orientations of a graph can have the same permanental polynomial. The next theorem gives a characterization on graphs whose all orientations have the same permanental polynomial.

An {\em $r$-matching} in a graph $G$ is a set of $r$ edges, no two of which have a vertex in common. The number of $r$-matchings in $G$
will be denoted by $p(G,r)$.

\begin{theorem}
\label{T:all_orientations}
All orientations $G^\sigma$ of a graph $G$ have the same permanental polynomial if and only if $G$ has no even cycles.
\end{theorem}

\begin{proof}
The sufficiency can be easily seen from Theorem~\ref{T:coeffient_skew}. Now we are going to prove the necessity of this theorem by contradiction. We assume that all orientations $G^\sigma$ have the same permanental polynomial, and $G$ contains an even cycle.

Let $2l$ be the smallest length of an even cycle in $G$. It is easy to see that the Sachs subgraphs of $\mathcal{E}\mathscr{U}_{2l}$ are $l$-matchings or $2l$-cycles. Let $G^\sigma$ be an orientation of $G$. By Theorem~\ref{T:coeffient_skew}, we have
\begin{equation}\label{E:skew}
a_{2l}=(-1)^{l}p(G,l)+2\sum_{C\in\mathscr{C}}(-1)^{c^-(C)},
\end{equation}
where $\mathscr{C}$ is the set of all $2l$-cycles in $G$.

For an edge $e$, denote by $n_{+}(e)$ (resp. $n_{-}(e)$) the number of evenly (resp. oddly) oriented $2l$-cycles in $G$ containing $e$. We claim that $n_{+}(e)$ = $n_{-}(e)$. Suppose, to the contrary, that $n_{+}(e)\neq n_{-}(e)$. Consider the new orientation of $G$ obtained from $G^\sigma$ by only reversing the orientation of $e$. Then in Eq. (\ref{E:skew}) the contribution from the $l$-matchings and those $2l$-cycles not containing $e$ will be unaffected, whereas the contribution from $2l$-cycles containing $e$ equals $2(n_{+}(e)-n_{-}(e))$ and will be negated. It follows that $a_{2l}$ will change under this new orientation of $G$, which contradicts all orientations $G^\sigma$ have the same permanental polynomial.

For $t\in\{1,2,\dots,2l\}$, denote by $n_{+}(e_1,e_2,\dots,e_t)$ (resp. $n_{-}(e_1,e_2,\dots,e_t)$) the number of evenly (resp. oddly) oriented $2l$-cycles in $G$ containing all of $e_1,e_2,\dots,e_t$.

We claim that for each $t\in\{1,2,\dots,2l\}$, $n_{+}(e_1,e_2,\dots,e_t)$ = $n_{-}(e_1,e_2,\dots,e_t)$ for all orientations $G^\sigma$ and all edges $e_1,e_2,\dots,e_t$. We proceed by induction on $t$.

The case $t=1$ has been proved as above. Assume that the claim holds for $t<2l$. Let $G^\sigma$ be an orientation of $G$. For edges $e_1,e_2,\dots,e_t,e_{t+1}$ of $G$, denote by $n_{+}(e_1,e_2,\dots,e_t,\overline{e_{{t+1}}})$ (resp. $n_{-}(e_1,e_2,\dots,e_t,\overline{e_{{t+1}}})$) the number of evenly (resp. oddly) oriented $2l$-cycles in $G$ containing $e_1,e_2,\dots,e_t$, but not $e_{t+1}$.
It is easy to see that
\[
n_{+}(e_1,e_2,\dots,e_t)=n_{+}(e_1,e_2,\dots,e_t,e_{{t+1}})+n_{+}(e_1,e_2,\dots,e_t,\overline{e_{{t+1}}}),
\]
and
\[
n_{-}(e_1,e_2,\dots,e_t)=n_{-}(e_1,e_2,\dots,e_t,e_{{t+1}})+n_{-}(e_1,e_2,\dots,e_t,\overline{e_{{t+1}}}).
\]

Consider the orientation $G^{\tilde{\sigma}}$ obtained from $G^{\sigma}$ by only reversing the orientation of $e_{t+1}$. Then
\[
\tilde{n}_{+}(e_1,e_2,\dots,e_t)=n_{-}(e_1,e_2,\dots,e_t,e_{t+1})+n_{+}(e_1,e_2,\dots,e_t,\overline{e_{{t+1}}}),
\]
and
\[
\tilde{n}_{-}(e_1,e_2,\dots,e_t)=n_{+}(e_1,e_2,\dots,e_t,e_{t+1})+n_{-}(e_1,e_2,\dots,e_t,\overline{e_{{t+1}}}).
\]

By the induction hypothesis, we have
\[
n_{+}(e_1,e_2,\dots,e_t)=n_{-}(e_1,e_2,\dots,e_t),
\]
and
\[
\tilde{n}_{+}(e_1,e_2,\dots,e_t)=\tilde{n}_{-}(e_1,e_2,\dots,e_t).
\]

Thus, we have
\begin{eqnarray*}
&\quad n_{+}(e_1,e_2,\dots,e_t,\overline{e_{{t+1}}})-n_{-}(e_1,e_2,\dots,e_t,\overline{e_{{t+1}}})\\
&= n_{-}(e_1,e_2,\dots,e_t,e_{t+1})-n_{+}(e_1,e_2,\dots,e_t,e_{t+1})\\
&= n_{+}(e_1,e_2,\dots,e_t,e_{t+1})-n_{-}(e_1,e_2,\dots,e_t,e_{t+1}).
\end{eqnarray*}

This gives $n_{+}(e_1,e_2,\dots,e_t,e_{t+1})$ = $n_{-}(e_1,e_2,\dots,e_t,e_{t+1})$. The claim holds from the principle of induction.

Let $C$ be an even cycle of length $2l$ and $e_1,e_2,\dots,e_{2l}$ the edges of $C$. By the above claim, we have $n_{+}(e_1,e_2,\dots,e_{2l})=n_{-}(e_1,e_2,\dots,e_{2l})$ for any orientation $G^{\sigma}$. This is impossible, since one side of the equality is 1, and the other is 0.
\end{proof}

The argument of the proof of Theorem~\ref{T:all_orientations} is similar to the one used to prove Theorem~4.2 in \cite{CaCF12}, which states that all the skew adjacency matrices of a graph $G$ have the same characteristic polynomial if and only if $G$ has no even cycles. It is easy to see that if $G$ has no even cycles, then each block of $G$ is either a complete graph $K_2$ on two vertices or an odd cycle.

The {\em matching polynomial}~\cite{GoGu81} of a graph $G$ on $n$ vertices is defined by
\[
\mu(G,x)=\sum_{r\ge0}(-1)^{r}p(G,r)x^{n-2r}.
\]

By Theorems~\ref{T:coeffient_skew} and~\ref{T:all_orientations}, we immediately obtain the following interesting result.

\begin{corollary}\label{C:oddgraph_root}
A graph $G$ has no even cycles if and only if $\pi(G^\sigma,x)$ = $\mu(G,x)$ for any orientation $G^\sigma$ of $G$.
\end{corollary}

\begin{remark}
{\rm
In \cite{ZhLL}, the authors showed that a non-empty graph $G$ has at least one complex permanental root. It is well-known that the roots of the matching polynomial of a graph are all real~\cite{GoGu81}. By Corollary \ref{C:oddgraph_root}, the permanental roots of any orientation $G^\sigma$ of a graph $G$ having no even cycles are all real.
}
\end{remark}

\section{Roots of the permanental polynomial}\label{S:permanental_root}
In~\cite{Bor85}, Borowiecki defined the \emph{per-spectrum} $S_p(G)$ of $G$ as the multiset of permanental roots of $G$. Analogously, the {\em per-spectrum} $S_p(G^\sigma)$ of $G^\sigma$ is defined as the multiset of permanental roots of $G^\sigma$. We use $S(G)$ to denote the adjacency spectrum of $G$. In this section, the relations among $S_p(G^\sigma)$, $S_p(G)$ and $S(G)$ are studied. Firstly, we present some necessary lemmas.

\begin{lemma}[\cite{MeRW81}]\label{T:Sachs_simplegraph}
Suppose $\pi(G,x)$ = $\mathrm{per}(xI-A(G))$ = $\sum_{k=0}^{n}a_{k}x^{n-k}$. Then
\[
a_k=(-1)^{k}\sum_{U\in \mathscr{U}_k}2^{c(U)},\quad 1\le k\le n.
\]
where  $\mathscr{U}_k$ is the set of all Sachs subgraphs of $G$ with exactly $k$ vertices, and $c(U)$ is the number of cycles in $U$.
\end{lemma}

\begin{lemma}[\cite{BoJo81}]\label{L:per_root_bipartite}
Let $G$ be a graph on $n$ vertices with
$\pi(G,x)=\sum_{k=0}^{n}a_{k}x^{n-k}$. Then $G$ is bipartite if and only if $a_{k}=0$ for all odd $k$.
\end{lemma}

\begin{lemma}[\cite{Bor85}]\label{L:Bor}
A graph $G$ satisfies $S_p(G)$ = $iS(G)$ ($i^2=-1$) if and only if $G$ is a bipartite graph without cycles of length $4l$ ($l\ge1$).
\end{lemma}

From relations between the coefficients and roots of a polynomial, the following lemma is easy to verify and the proof is omitted.
\begin{lemma}\label{P:per_root}
Suppose $\pi(G,x)$ = $\sum\limits_{k=0}^na_{k}(G)x^{n-k}$ and $\pi(G^\sigma,x)$ = $\sum\limits_{k=0}^na_k(G^\sigma)x^{n-k}$. Then $S_p(G^\sigma)$ = $iS_p(G)$ ($i^2=-1$) if and only if $a_k(G^\sigma)$ = $a_{k}(G)$ = 0 if $k$ is odd, $a_k(G^\sigma)$ = $a_{k}(G)$ if $k\equiv0\pmod 4$, and $a_k(G^\sigma)$ = $-a_{k}(G)$ if $k\equiv2\pmod 4$.
\end{lemma}

\begin{theorem}\label{T:bipa_per_root}
There exists an orientation $G^\sigma$ of a graph $G$ such that $S_p(G^\sigma)=iS_p(G)$ ($i^2=-1$) if and only if $G$ is bipartite.
\end{theorem}

\begin{proof}
Suppose that $\pi(G,x)$ = $\sum\limits_{k=0}^na_{k}(G)x^{n-k}$ and $\pi(G^\sigma,x)$ = $\sum\limits_{k=0}^na_k(G^\sigma)x^{n-k}$. If there exists an orientation $G^\sigma$ of $G$ such that $S_p(G^\sigma)=iS_p(G)$, then from Lemma \ref{P:per_root}, $a_k(G)$ = $a_{k}(G^\sigma)$ = 0 for all odd $k$. By Lemma~\ref{L:per_root_bipartite}, $G$ is bipartite.

Conversely, we assume that $G$ is bipartite and $(X,Y)$ is a bipartition of $G$. Orient $G$ by directing all edges of $G$ toward $Y$. Denote this orientation by $G^\sigma$. If $G$ contains no cycles, then by Theorem~\ref{T:coeffient_skew} and Lemma~\ref{T:Sachs_simplegraph} we have $a_{2l}(G^\sigma)$ = $(-1)^{l} p(G,l)$ = $(-1)^{l}a_{2l}(G)$. So we assume that $G$ contains an even cycle. Let $C_{2l}$ be an even cycle in $G$ of length $2l$. Then $C_{2l}$ is oddly oriented if and only if $l$ is odd.

By Lemma~\ref{T:Sachs_simplegraph} and Theorem~\ref{T:coeffient_skew}, we have
\[
a_{2l}(G)=\sum_{U\in\mathscr{U}_{2l}}2^{c(U)}=p(G,l)+\sum_{U\in\mathscr{U}_{2l} \atop c(U)>0}2^{c(U)},
\]
and
\[
a_{2l}(G^\sigma)=(-1)^{l}p(G,l)+\sum_{U\in\mathcal{E}\mathscr{U}_{2l} \atop c(U)>0}
(-1)^{m(U)+c^{-}(U)}2^{c(U)}.
\]

Since $G$ is bipartite, it follows that $\mathscr{U}_{2l}$ = $\mathcal{E}\mathscr{U}_{2l}$. Let $U$ be a Sachs subgraph of $G$ on $2l$ vertices containing at least one cycle. Let $H_1$, $H_2$, $\dots$, $H_{c(U)}$ be all the cycles of $U$. Without loss of generality, we assume that $H_1$, $H_2$, $\dots$, $H_{c^{-}(U)}$ are oddly oriented relative to $G^\sigma$. Suppose that $l(H_j)=2(2s_j+1)$ for $j=1,2,\dots,c^{-}(U)$ and $l(H_j)=2(2s_j)$ for $j=c^{-}(U)+1,\dots,c(U)$, where $l(H_j)$ denotes the length of cycle $H_j$. Thus
\[
2m(U)+2\sum_{j=1}^{c^{-}(U)}(2s_j+1)+2\sum_{j=c^{-}(U)+1}^{c(U)}(2s_j) = 2l.
\]
It follows that $m(U)+c^{-}(U)\equiv l\pmod{2}$. Therefore,
\[
a_{2l}(G^\sigma)=(-1)^{l}p(G,l)+\sum_{U\in\mathcal{E}\mathscr{U}_{2l} \atop c(U)>0}(-1)^{l}2^{c(U)}
=(-1)^{l}a_{2l}(G).
\]

Since $G$ is bipartite, we have $a_{2l+1}(G)$ = 0 from Lemma~\ref{L:per_root_bipartite}. By Theorem~\ref{T:coeffient_skew}, we have $a_{2l+1}(G^\sigma)$ = 0. Thus $a_{2l+1}(G^\sigma)$ = $a_{2l+1}(G)$ = 0. Therefore $S_p(G^\sigma)=iS_p(G)$ from Lemma \ref{P:per_root}.
\end{proof}

By Theorems~\ref{T:all_orientations} and~\ref{T:bipa_per_root}, we immediately obtain the following corollary.

\begin{corollary}\label{T:forest_root}
For any orientation $G^\sigma$ of $G$, $S_p(G^\sigma)=iS_p(G)$ ($i^2=-1$) if and only if $G$ is a forest.
\end{corollary}

The next result gives a characterization on graphs satisfying $S_p(G^\sigma)$ = $S(G)$.

\begin{corollary}\label{C:forest_root}
For any orientation $G^\sigma$ of $G$, $S_p(G^\sigma)$ = $S(G)$ if and only if $G$ is a forest.
\end{corollary}

\begin{proof}
If $G$ is a forest, then by Corollary~\ref{T:forest_root} and Lemma~\ref{L:Bor}, $S_p(G^\sigma)$ = $-S(G)$ for any orientation $G^\sigma$ of $G$. Since the adjacency spectrum of a bipartite graph is symmetric with respect to the origin, that is $S(G)=-S(G)$, we have $S_p(G^\sigma)$ = $S(G)$. Conversely, we assume that $S_p(G^\sigma)$ = $S(G)$ for any orientation $G^\sigma$ of $G$. From Theorem~\ref{T:all_orientations}, we know that $G$ has no even cycles. By Theorem~\ref{T:coeffient_skew}, $a_k(G^\sigma)=0$ for all odd $k$. Since $S_p(G^\sigma)$ = $S(G)$, we have $b_k(G)=0$ for all odd $k$, where $b_k(G)$ is the coefficient of $x^{n-k}$ in the characteristic polynomial of $G$. It implies that $G$ is bipartite. Thus $G$ is a forest.
\end{proof}

\begin{remark}
\rm
If $G$ is a bipartite graph without cycles of length $4l$ ($l\ge1$), then by Lemma~\ref{L:Bor} and Theorem~\ref{T:bipa_per_root} there exists an orientation $G^\sigma$ of $G$ such that $S_p(G^\sigma)$ = $S(G)$. The converse of this statement is not true. For example, we consider the cycle $C_4$ of length 4. It is easy to see that $S(C_4)$ = $\{-2,0,0,2\}$. If $C_4$ is oddly oriented, then by Theorem~\ref{T:coeffient_skew}, we have $\pi(C_4^\sigma,x)$ = $x^4-4x^2$ and $S_p(C_4^\sigma)$ = $\{-2,0,0,2\}$.
\end{remark}


\end{document}